\documentclass[twoside,11pt]{article}
\usepackage{graphicx,amscd,amsfonts,amsmath,amsthm,amssymb,latexsym,color}
\usepackage{epsfig,float,epstopdf,array,bm,cite}
\usepackage[bookmarksnumbered, colorlinks,plainpages]{hyperref}
\usepackage{longtable,rotating,multirow}
\usepackage{varwidth}

 \newtheorem{theorem}{\bf Theorem}

\newcommand{\Title}[1]{\title{\bf#1}}
 \newcommand {\Author}[1]{\author{\small\bf #1 \vspace{1mm} \addr \emails}}
 \newcommand{\Address}[1]{\def\addr{\\ \vspace{2mm}{\small\textit{ \begin{varwidth}{.9\textwidth} \centering #1\end{varwidth} }}\\}
 }
 \newcommand{\Markboth}[2]{\pagestyle{myheadings}
 	\markboth{#1}{ #2}}
 
 \newcommand{\Email}[1]{\newcommand{\emails}{ \begin{varwidth}{.9\textwidth} \centering{\small \textit{Email(s): #1}} \end{varwidth}
 \vspace{-1cm}
}}
 
 \newcommand{\correspond}{\thanks{\noindent 
 		Corresponding author. \newline
 		\scriptsize{Received:  / Revised:  / Accepted: } \newline
 		DOI:  \newline\newline
 		{\small\bf \copyright~2021 University of Guilan \hfill \bf \url{http://jmm.guilan.ac.ir}}}~~}
 	\newcommand{\Abstract}[1]{
 		\date{}
 		\maketitle 
 		\noindent\rule{6.29in}{0.02in}\\
 		{\bf Abstract.\,}{#1 \\}}
 	\newcommand{\Keywords}[1]{\noindent{\it \footnotesize Keywords}: {\small #1}\\  }
 	\newcommand{\AMS}[2]{\noindent
 		\noindent{\it \footnotesize AMS Subject Classification #1}:~#2\\
 		\noindent\rule{6.29in}{0.02in}
 		\thispagestyle{empty}}
 
 \newcommand{\Vol}[1]{\newcommand{\VOLUME}{#1}}
 \newcommand{\Num}[1]{\newcommand{\NUMBER}{#1}}
 \newcommand{\Year}[1]{\newcommand{\YEAR}{#1}}
 \newcommand{\PP}[1]{\newcommand{\PAGE}{~#1}}
 \newcommand{\spacelogo}[1]{\newcommand{\SpaceLogo}{#1}}

\Vol{8}
\Num{4}
\Year{2020}
\PP{455--477}

\spacelogo{220}
\begin{document}
\Title{ Optimal fire allocation in a combat model of mixed NCW type}

\Author{My A. Vu$^\dag$\correspond,  Nam H. Nguyen$^{\ddag}$, Hanh Le T. Nguyen$^{\S}$, Anh N. Ta$^\heartsuit$, Mong H. Nguyen$^\spadesuit$} 
%\Author{}
\Address{$^{\dag}$,$^{\ddag}$,$^\heartsuit$,$^\spadesuit$ Department of Mathematics, Faculty of Information Technology, Le Quy Don Technical University, Ha Noi, Vietnam\\
         $^{\S}$Faculty of Fundamental Sciences, University of Economics and Industrial Technology, Ha Noi, Vietnam}

\Email{myva@lqdtu.edu.vn, nguyenhongnam1977@gmail.com, nthle@uneti.edu.vn, tangocanh@gmail.com, nghm06@yahoo.com}
\Markboth{My VA, Nam NH, Hanh Le NT, Anh TN, Mong NH}{Optimal fire allocation in mixed NCW model}

\Abstract{In this work, we introduce a nonlinear Lanchester model of NCW-type and study a problem of finding the optimal fire allocation for this model. A Blue party $B$ will fight against a Red party consisting of $A$ and $R$, where $A$ is an independent force and $R$ fights with supports from a supply unit $N$. A battle may consist of several stages but we consider the problem of finding optimal fire allocation for $B$ in the first stage only. Optimal fire allocation is a set of three non-negative numbers whose sum equals to one, such that the remaining force of $B$ is maximal at any instants. In order to tackle this problem, we introduce the notion of \textit{threatening rates} which are computed for $A, R, N$ at the beginning of the battle. Numerical illustrations are presented to justify the theoretical findings.
}

\AMS{2010}{34A34, 65L05.}
 \section{Introduction}
\label{S:1}
%%%%%%%%%%%%%%%%%%%%%%%%
	In 1916, Lanchester \cite{Lan} introduced a mathematical model for
a battle in the form of a system of differential equations. 
This model has been extended and generalized in various ways, such as guerilla model by Deitchman \cite{Dei}, guerilla model with intelligence by Schaffer \cite{Sch} and  Schreiber
\cite{Sch1}, counter terrorism model by Kaplan, Kress and Szechtman (KKS) \cite{Kap},\cite{Kre}. There are several problems involving these models, among of which are the problem of optimal fire allocation with number of troops being objective function. This problem has been investigated in various scenarios by Taylor \cite{Tay}, Lin and Mackay \cite{Mac}. In these works, the role of military supply, however, has not been studied thoroughly. In a combat, the victory of either party is not only decided by the armed forces but also by their supply units. In many historical battles, firepower was not only aimed at the direct rivals but also their supply units, see \cite{Welborn}. \\
In 2017, Kim and his colleagues \cite{Kim2017} considered a Lanchester's model where Blue force $B$ fights against Red force $R$ supported by a supply unit $N$ and called this a NCW model. This model can be denoted by $\left( B \textit{ vs } (R,N)\right)$. Kim considered fire allocations in the form of piecewise constant functions and derived optimal fire allocation so that number of $B$'s troop is always at its possible maximum. In this paper, we extend Kim's model by considering a model of $\left( B \textit{ vs } \{(R,N), A\}\right)$ where $A$ stands for an independent force. Let us refer this model as mixed NCW. For this model, we also consider the problem of finding optimal fire allocation of $B$ so that its remaining troop at any instant is maximal. By a different approach with the help of \textit{"threatening rates"}, we managed to derive the optimal fire allocation for $B$. 
In Lanchester's model using system of differential equations,
the decreasing rate of troops of a force is computed by attrition
rate of its rival force multiplied by the rival's number of troops.
In our model, attrition rate of $R$ is assumed to
be a linear function of the amount of $N$'s troop and this
supply unit can also be attrited by $B$. The resulting model is of non-linear Lanchester type. Let us recall that in classical non-linear Lanchester's model, the supply units have not been taken into account but only the armed forces.\\
In any battle of the type "one against many", the strategy will be intuitively derived as follows: focus all firepower to the entity possessing the "greatest threat". In order to quantify "threats" posed by the entities $R,\, N, \,A$ to $B$, we introduce the notion of "threatening rates" which are computed for each entity $R,\, N, \,A$ by involving entities number of troops, their attrition rates against $B$ and $B$'s attrition rates. 
The fire allocation of $B$ is assumed to be constant for a certain period of time, which we call "stage". The first stage is the period from the beginning of the battle to the instant when one entity is extirpated - its number of troops reaches zero. During this stage, the fire allocation is kept constant and the battle moves to a new stage if one (or more) entity is eliminated. This choice of fire allocation is meaningful since in planning phase of a battle, this choice simplifies the logistic operations. Moreover, capturing the states of the battlefield and altering the fire allocation accordingly is not an easy task. 
By using the threatening rates, we justify the intuitive strategy mentioned above. Thus, the optimal fire allocation for $B$ is focusing all its firepower to the entity possessing the greatest threatening rate. Several numerical examples are included to illustrate the theoretical findings. \\
The rest of the paper is organized as follows.  Section 2 is devoted to inroduce our model and to investigate the optimization problem for this model.
Numerical experiments are presented in Section 3 to
illustrate the theoretical results. Conclusion and some possible further developments are discussed in the last section.
%%%%%%%%%%%%%%%%%%%%%%%%
\section{Main results}
\subsection{Non-linear Lanchester model of mixed NCW type}
Let us consider a battle where $B$ fights against $((R,N),A)$. Our model is called non-linear Lanchester of NCW type. NCW stands for "Network Centric Warfare", which is a novel notion of modern warfare. For a more detailed explanation of this notion, the reader is referred to \cite{Alb, Tun, Tun1, Owen}. A simple diagram for the model is represented in Figure \ref{hinh1}.
\begin{figure}[h]
\centering
\includegraphics[width=0.6\linewidth]{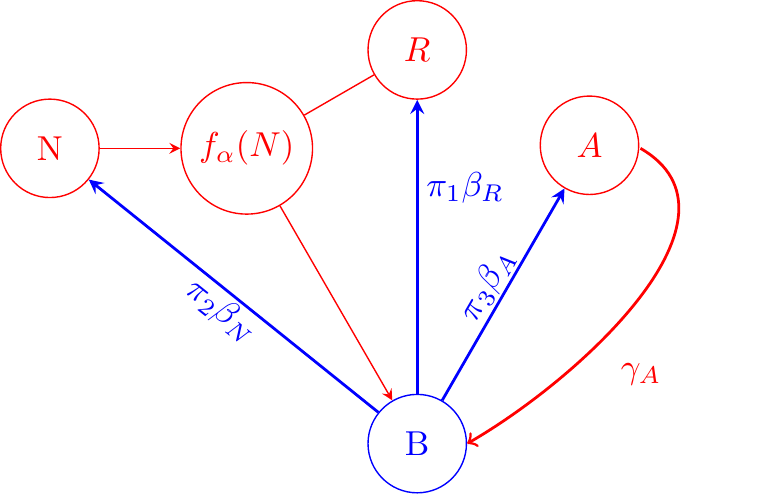}
\caption{Diagram of $\left( B \textit{ vs } \{(R,N), A\}\right)$ model.}
\label{hinh1}
\end{figure}
Before formulating the model, let us denote:
\begin{itemize}
\item by ${{\beta }_{R}}:$ attrition rate of $B$ against $R$.
\item by ${{\beta }_{A}}:$ attrition rate of $B$ against $A$.
\item by ${{\beta }_{N}}:$ attrition rate of $B$ against $N$.
\item ${{f}_{\alpha }}(N):$ attrition function of $N$ complementing $R^i$ to $B$.
 \item $\Pi =({{\pi }_{1}},{{\pi }_{2}},{{\pi }_{3}}):$ fire allocation of $B$ against $R$, $N$, $A$, respectively.
 \item $\alpha _{c}^{N}:$ fully-connected attrition rate of $R$ against $B$. 
 \item $\alpha _{d}^{N}:$ disconnected attrition rate of $R$ against $B$, $\left( \alpha _{d}^{N}\le \alpha _{c}^{N} \right).$
\item ${{\gamma }_{A}}:$ attrition rate of $A$ against $B$.
\item $R_0,\, N_0,\, A_0$ are the initial numbers of troops of $R,\, N,\, A$, respectively.
\end{itemize}
The fire allocation of $B$ is sought in the set 
\begin{equation}\label{Pi}
\mathcal{P}=\{\Pi =\left( \pi _{1},\pi _{2},\pi _{3} \right) | \pi_{1}+\pi_{2}+\pi _{3}=1,\, \pi _{i} \in \left[ 0;1 \right],\, \pi _i 's \text{ are constants },\, i=1,2,3\}.
\end{equation}  
The attrition function of $R$ against $B$ is assumed to be a linear function of $N$: 
\begin{equation} \label{lin_func}
{{f}_{\alpha }}\left( N \right)=\alpha _{d}^{N}+\left( \alpha _{c}^{N}-\alpha _{d}^{N} \right)\frac{N}{{{N}_{0}}},
\end{equation} 
 When $N={{N}_{0}}$, ${{f}_{\alpha }}\left( N \right)=\alpha _{c}^{N},$ in other words, $R$ and $N$ are fully connected. When $N$ is totally eliminated by $B,$  $N=0$, $R$ and $N$ are totally disconnected and ${{f}_{\alpha }}\left( N \right)=\alpha _{d}^{N}.$\\
Our model can be described by a system of differential equations defined as follows: 
	\begin{equation}\label{pt1}
	\left \{
 \begin{aligned}
	 \frac{dB}{dt}&=-\left[ \alpha _d^N+\left( \alpha _c^N-\alpha _d^N \right)\frac{N}{N_0} \right] R-\gamma _A A,  \\
   	\frac{dR}{dt}&=-\pi_1\beta_RB,  \\
  	 \frac{dN}{dt}&=-\pi_2\beta _NB,  \\
   	\frac{dA}{dt}&=-\pi_3\beta_AB.  
   	\end{aligned}
 \right.
   \end{equation}                                                  
\subsection{Optimal fire allocation}
For the model described by \eqref{pt1}, we consider the first stage of the conflict and study the problem of maximizing the remaining troops of $B$ at any instant $t$ with respect to fire allocations belonging to the set \eqref{Pi}. Let us introduce the notion of “threatening rates”, which are denoted by ${{b}_{1}},\,{{b}_{2}},\,{{b}_{3}}$, and computed as follows: 
\begin{equation}\label{rates}
	\left \{
 \begin{aligned}
{{b}_{1}}&=\alpha _{c}^{N}{{\beta }_{R}},\\
{{b}_{2}}&=\frac{{{\beta }_{N}}\left( \alpha _{c}^{N}-\alpha _{d}^{N} \right){{R}_{0}}}{{{N}_{0}}},\\
{{b}_{3}}&={{\gamma }_{A}}{{\beta }_{A}}.
\end{aligned}
 \right.
   \end{equation}      
These quantities represent the level of harm which $R,\,N,\,A$ can possibly inflict on $B$. By invoking these rates, the optimal fire allocation of $B$ is derived in the following theorem.
\begin{theorem}
Among all the fire allocations of the set \eqref{Pi} for the first stage, the optimal one for $B$  is as follows:
\begin{equation*}
	\Pi^*=\left \{
 \begin{aligned}
 \left( 1,0,0 \right) \text{  if }\,&\left( {{b}_{1}}\geq{{b}_{2}} \right)\wedge \left( {{b}_{1}}\geq{{b}_{3}} \right),  \\
   \left( 0,1,0 \right) \text{ if } \,&\left( {{b}_{2}}\geq{{b}_{1}}\geq{{b}_{3}} \right)\vee \left( {{b}_{2}}\geq{{b}_{3}}\geq{{b}_{1}} \right),  \\
   \left( 0,0,1 \right)\text{ if } \,&\left( {{b}_{3}}\geq{{b}_{1}}\geq{{b}_{2}} \right)\vee \left( {{b}_{3}}\geq{{b}_{2}}\geq{{b}_{1}} \right).  
\end{aligned}
 \right.
   \end{equation*}      
\end{theorem}
\begin{proof}
Let  $X\left( t \right)=\int\limits_{0}^{t}{B\left( s \right)ds\Rightarrow }{X}'\left( t \right)=B\left( t \right)$. We have then
\begin{equation}\label{pt2}
   {X}''\left( t \right)={B}'\left( t \right)=-\left[ \alpha _{d}^{N}+\left( \alpha _{c}^{N}-\alpha _{d}^{N} \right)\frac{N}{{{N}_{0}}} \right]R-{{\gamma }_{A}}A.
\end{equation}     
Integrating both sides of the second equation of \eqref{pt1} yields
$$\int\limits_{0}^{t}{dR=-\int\limits_{0}^{t}{{{\pi }_{1}}{{\beta }_{R}}B\left( s \right)ds}}\Rightarrow R\left( t \right)-R\left( 0 \right)=-{{\pi }_{1}}{{\beta }_{R}}X\left( t \right).$$
From this we obtain 
\begin{equation}\label{pt3}
R=-{{\pi }_{1}}{{\beta }_{R}}X\left( t \right)+{{R}_{0}}.
\end{equation}    
By analogous computations, we get 
\begin{eqnarray}
N&=&-{{\pi }_{2}}{{\beta }_{N}}X\left( t \right)+{{N}_{0}}, \label{pt4}    \\        
A&=&-{{\pi }_{3}}{{\beta }_{A}}X\left( t \right)+{{A}_{0}}.  \label{pt5}
\end{eqnarray}
where $R_0, N_0, A_0$ are the corresponding numbers of troops of $R, N, A$ at the beginning of the combat. 
Substituting $R,N,A$   in \eqref{pt3},\eqref{pt4},\eqref{pt5} into \eqref{pt2} yields:
 \begin{equation}\label{pt6}
{X}''\left( t \right)=-{{C}_{1}}{{X}^{2}}\left( t \right)+{{C}_{2}}X\left( t \right)-{{C}_{3}},
\end{equation}    
where:
\begin{eqnarray*}
{{C}_{1}}&=&\frac{{{\pi }_{1}}.{{\pi }_{2}}.{{\beta }_{R}}.{{\beta }_{N}}.\left( \alpha _{c}^{N}-\alpha _{d}^{N} \right)}{{{N}_{0}}},\\
{{C}_{2}}&=&\frac{{{\pi }_{2}}{{\beta }_{N}}\left( \alpha _{c}^{N}-\alpha _{d}^{N} \right){{R}_{0}}+{{\pi }_{1}}{{\beta }_{R}}\alpha _{c}^{N}{{N}_{0}}}{{{N}_{0}}}+{{\gamma }_{A}}{{\pi }_{3}}{{\beta }_{A}},\\
 {{C}_{3}}&=&\alpha _{c}^{N}{{R}_{0}}+{{\gamma }_{A}}{{A}_{0}}.
\end{eqnarray*}
Multiplying both sides of \eqref{pt6} by $d{X}'\left( t \right)$ and integrating, we obtain
\begin{eqnarray*}
{X}'\left( t \right)&=&B\left( t \right)  \\
                         &=&\sqrt{-\frac{2}{3}{{C}_{1}}{{X}^{3}}\left( t \right)+{{C}_{2}}{{X}^{2}}\left( t \right)-2{{C}_{3}}X\left( t \right)+{{C}_{4}}}.
\end{eqnarray*}
  It can be seen that ${{C}_{1}},{{C}_{2}}$ are nonnegative and they depend solely on the fire allocation $\Pi$, $C_3$ is independent of $\Pi$ and ${{C}_{4}}$ is an integral constant. In the stage under consideration, at any time $t=t_0$, $X(t_0)$ is a positive constant. Therefore, it is observed that the smaller $C_1$ and the greater $C_2$ make $B(t_0)$ greater. We then need to solve the following multi-objective optimization problem:
$$
\min \limits _{\Pi \ in \mathcal{P}} \{(C_1; -C_2)\}.
$$
Let us denote
\begin{eqnarray*}
a&=&\frac{{{\beta }_{R}}{{\beta }_{N}}\left( \alpha _{c}^{N}-\alpha _{d}^{N} \right)}{{{N}_{0}}},\\
 {{b}_{1}}&=&\alpha _{c}^{N}{{\beta }_{R}},\\
 {{b}_{2}}&=&\frac{{{\beta }_{N}}\left( \alpha _{c}^{N}-\alpha _{d}^{N} \right){{R}_{0}}}{{{N}_{0}}},\\
 {{b}_{3}}&=&{{\gamma }_{A}}{{\beta }_{A}},\\
 x&=&\pi_1,y=\pi_2, z=\pi_3.
 \end{eqnarray*}
 The problem now becomes: 
\begin{equation}\label{pt7}
\min \{\left( axy, -(b_1x+b_2y+b_3z)\right) \} \text{  s.t. }\left\{ \begin{aligned}
  & 0\le x,y,z\le 1, \\ 
 & x+y+z=1. \\ 
\end{aligned} \right.
\end{equation}
To tackle this problem, we will use weighting method (see \cite[Section 3.1]{Kai1998}). We thus set 
$$F_\lambda\left( x,y,z \right)=\lambda \left( axy \right)-\left( 1-\lambda  \right)\left( {{b}_{1}}x+{{b}_{2}}y+{{b}_{3}}z \right),$$
and obtain:
\begin{equation}\label{pt8}
\min F_\lambda(x,y,z) \text{  s.t. }\left\{ \begin{aligned}
  & 0\le x,y,z\le 1, \\ 
 & x+y+z=1, \\ 
 &0\le \lambda \le 1.
\end{aligned} \right.
\end{equation}
Substituting $x=1-y-z$, we now get the problem: 
\begin{equation}
\label{pt9}
\begin{aligned}
&\min \{\lambda ( 1-y-z )ay-( 1-\lambda  )( {{b}_{1}}+( {{b}_{2}}-{{b}_{1}} )y+( {{b}_{3}}-{{b}_{1}} )z )  \}\\
&\text{s.t. } \left\{ \begin{aligned}
  &  y,z \ge 0, \\ 
 & y+z\le 1, \\ 
 &0\le \lambda \le 1.
\end{aligned} \right.
\end{aligned}
\end{equation}
Let us consider the following five cases:
\begin{enumerate}
\item If $b_2>b_3 >b_1,$ since $\lambda \left( 1-y-z \right)ay\ge 0,$ it follows: 
\begin{eqnarray*}
   \min F_\lambda &\ge & -\left( 1-\lambda  \right)\left( {{b}_{1}}+\left( {{b}_{2}}-{{b}_{1}} \right)y+\left( {{b}_{3}}-{{b}_{1}} \right)z \right)  \\
   &\ge& -\left( 1-\lambda  \right)\left( {{b}_{1}}+\left( {{b}_{2}}-{{b}_{1}} \right)\left( y+z \right) \right)  \\
   &\ge & -\left( 1-\lambda  \right){{b}_{2}}=F_\lambda\left( 0,1,0 \right).  
\end{eqnarray*}
\item If ${{b}_{2}}>{{b}_{1}}>{{b}_{3}},$ the problem becomes:
\[\min\left\{ \lambda \left( 1-y-z \right)ay+\left( 1-\lambda  \right)\left( {{b}_{1}}-{{b}_{3}} \right)z-\left( 1-\lambda  \right)\left( {{b}_{1}}+\left( {{b}_{2}}-{{b}_{1}} \right)y \right) \right\}. \]
Therefore: 
$$\min F_\lambda\ge -\left( 1-\lambda  \right){{b}_{2}}=F_\lambda\left( 0,1,0 \right).$$
\item  If ${{b}_{3}}>{{b}_{2}}>{{b}_{1}}$
we obtain:
\begin{eqnarray*}
   \min F_\lambda&\ge& -\left( 1-\lambda  \right)\left( {{b}_{1}}+\left( {{b}_{3}}-{{b}_{1}} \right)\left( y+z \right) \right)  \\
   &\ge& -\left( 1-\lambda  \right){{b}_{3}}=F_\lambda \left( 0,0,1 \right).  
\end{eqnarray*}
\item If ${{b}_{3}}>{{b}_{1}}>{{b}_{2}}$
the problem is now:
\[ \min \left\{ \lambda \left( 1-y-z \right)ay+\left( 1-\lambda  \right)\left( {{b}_{1}}-{{b}_{2}} \right)y-\left( 1-\lambda  \right)\left( b+\left( {{b}_{3}}-b \right)z \right) \right\}.  
\]
We yield that: 
$$\min F_\lambda \ge -\left( 1-\lambda  \right){{b}_{3}}=F_\lambda\left( 0,0,1 \right).$$
\item If $b_2<b_1,\, b_3<b_1$, the problem turns out to be:
$$\min \left\{ \lambda \left( 1-y-z \right)ay+\left( 1-\lambda  \right)\left( \left( {{b}_{1}}-{{b}_{2}} \right)y+\left( {{b}_{1}}-{{b}_{3}} \right)z \right)-\left( 1-\lambda  \right){{b}_{1}} \right\}.$$
And, it follows that $\min F_\lambda \ge -\left( 1-\lambda  \right)b=F_\lambda\left( 1,0,0 \right). $
\end{enumerate}
%We may verify the fact "when $C_1$ attains its minimum, $C_2$ reach its maximum" as follows. Let us consider the case $b_2 > b_3 >b_1$ (other cases should be treated similarly). We have $\min F_\lambda = F_\lambda (0,1,0)$ or $\pi_1 = \pi_3=0, \pi_2 =1$ and it follows that $C_1= \frac{\pi_1 \pi_2 \beta _R \beta _N (\alpha _c^N-\alpha_d^N)}{N_0}=0$- its minimum. Similarly, ${{C}_{2}}=\frac{{{\pi }_{2}}{{\beta }_{N}}\left( \alpha _{c}^{N}-\alpha _{d}^{N} \right){{R}_{0}}+{{\pi }_{1}}{{\beta }_{R}}\alpha _{c}^{N}{{N}_{0}}}{{{N}_{0}}}+{{\gamma }_{A}}{{\pi }_{3}}{{\beta }_{A}}=b_2\pi_2+b_1\pi_1+b_3\pi_3$ and $C_2$ attains its maximum $b_2=\frac{\beta_N (\alpha _c^N-\alpha_d^N)R_0}{N_0}$ when $\pi_1=\pi_3 =0, \pi_2=1$.
\end{proof}
In principle, the battle can be divided into three stages. In the first stage, due to the computed threatening rates ${{b}_{1}},\,{{b}_{2}},\,{{b}_{3}}$ the Blue force $B$ will focus all its firepower to one of the entities $R,N,A.$ When one of the entities is eliminated, the second stage begins. When one of the remaining two forces is extinguisged, the third stage follows. However, if $R$ and $A$ are out of the picture by the end of the second stage, $B$ will be no longer attrited and the battle finishes. 
As stated above, for the first stage, our results apply. For the second stage, one may use the results by Kim et. al. \cite{Kim2017} or by Lin and Mackay \cite{Mac}. 
Thus, by the first case in our proof, in order that $B\left( t \right)$ is always maximal, for the first stage $B$ should focus its firepower to $N$. After the conclusion of $N$, the second stage begins, where the results of Lin and Mackay apply, thus $B$ concentrates its firepower on $R$ or $A$.

The second case in our proof is explained analogously. 

By the third case in our proof, $B$ will focus its firepower on $A$. The second stage in this case has been considered in \cite{Kim2017}. 
The fourth case is analyzed similarly as the third case.
Strategy in the fifth case can be explained as follows: for the first stage, $B$ concentrates on $R.$ $B$ then turn its firepower to $A$ for the second stage. And the battle ends when $A$ is totally annihilated. 
\section{Numerical experiments}
\subsection{Case 1} Let us consider equation \eqref{pt1} with coefficients given by
\begin{table}[!h]
\centering
\begin{tabular}{|c|c|c|c|c|c| }
\hline
$\alpha_c^N$ & $\alpha_d^N$&$\gamma_A$&$\beta_R$&$\beta_N$&$\beta_A$
\\
\hline
$0.4$&$0.15$&$0.2$&$0.5$&$0.3$&$0.2$\\
\hline
\end{tabular}
\caption{Parameters for Case 1.}
\end{table}\\
together with the following initial conditions : ${{B}_{0}}=170;{{R}_{0}}=120;\,\,{{N}_{0}}=20;\,{{A}_{0}}=50$. Threatening rates are thus computed as: ${{b}_{1}}=0.2;\,\,\,\,{{b}_{2}}=0.45;\,\,\,\,{{b}_{3}}=0.04.$ It is obvious that ${{b}_{2}}>{{b}_{1}}>{{b}_{3}}$, so the optimal fire allocation for the first stage is given by  $\left( 0,1,0 \right)$. 
By using results of Lin and Mackay, the optimal strategy for the whole battle is given by: 
\[{{\Pi }^{*}}=\left( 0,1,0 \right)\to \left( 1,0,0 \right)\to \left( 0,0,1 \right).\] 
This strategy should be interpreted as follows: for the first stage, $B$ will focus all its firepower to $N$. For the second stage, $B$ will concentrate on $R$ since $b_1 >b_2$. 
In order to contrast with the optimal strategy, we use ${{\Pi }_{1}}=\left( 1,0,0 \right) \to \left( 0,0,1 \right).$ This strategy is explained as: $B$ will focus all its firepower on $R$ in the first stage; for the second stage, it will concentrate on $A.$ The simulation results show that $B$ still win the battle with this strategy. However its amount of troops at any instant is always smaller than itself using the optimal one. Amounts of troops of $B$ using both stragtegies are represented in Figure \ref{th1}.
\begin{figure}[!h]
\centering\includegraphics[width=1.0\linewidth]{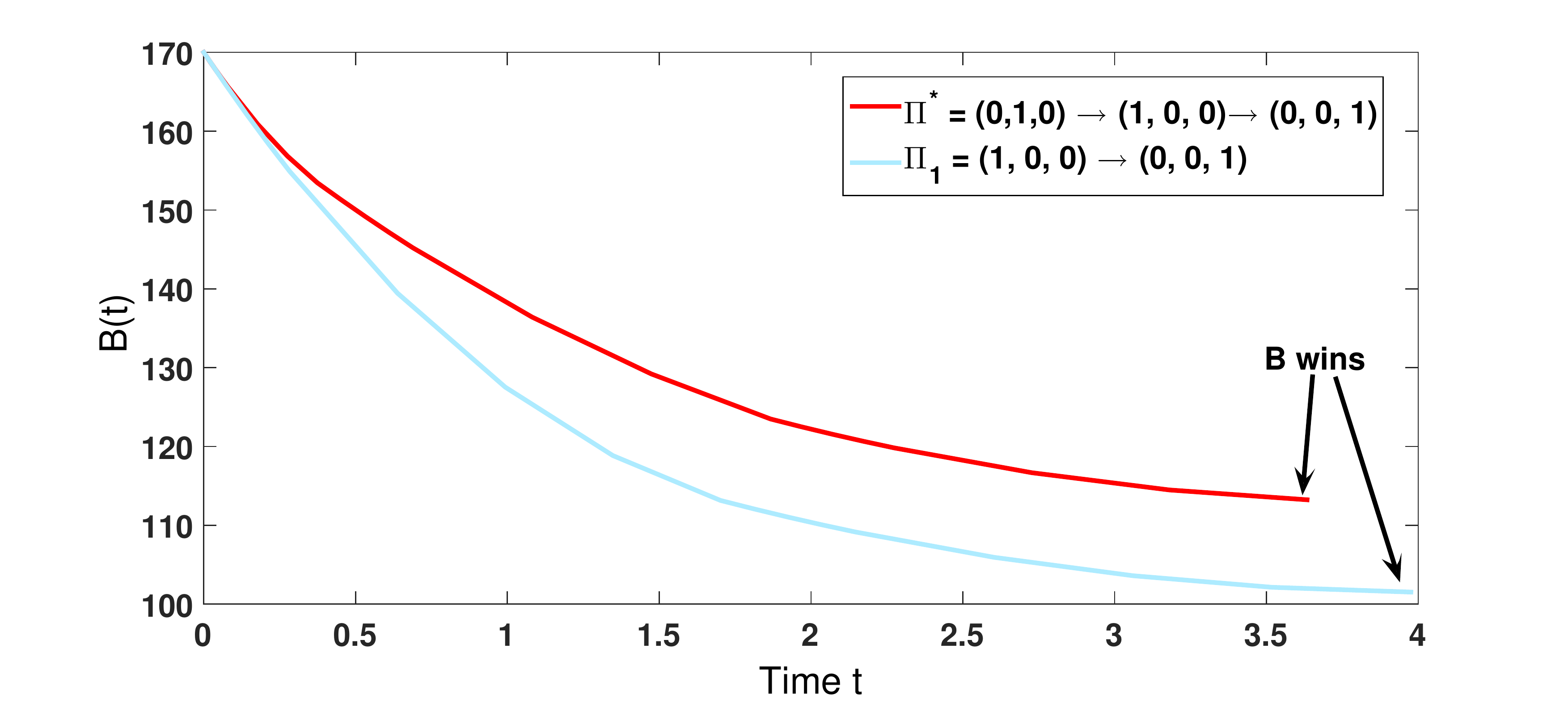}
\caption{Results for Case 1.} 
\label{th1}
\end{figure}

\subsection{Case 2}
We now will investigate model \eqref{pt1} with the following parameters 
\begin{table}[!h]
\centering
\begin{tabular}{|c|c|c|c|c|c| }
\hline
$\alpha_c^N$ & $\alpha_d^N$&$\gamma_A$&$\beta_R$&$\beta_N$&$\beta_A$
\\
\hline
$0.4$&$0.15$&$0.2$&$0.5$&$0.2$&$0.2$\\
\hline
\end{tabular}
\caption{Parameters for Case 2}
\end{table}\\
together with these initial conditions: ${{B}_{0}}=170,\,{{R}_{0}}=120,\,{{N}_{0}}=50,\,{{A}_{0}}=50.$
Threatening rates are thus computed as: ${{b}_{1}}=0.2;\,\,\,\,{{b}_{2}}=0.12;\,\,\,\,{{b}_{2}}=0.04.$ Since ${{b}_{2}}>{{b}_{1}}>{{b}_{3}}$, the optimal fire allocation for the first stage is ${{\Pi }^{*}}=\left( 1,0,0 \right).$ Among three strategies we chose to compare, only ${{\Pi }_{3}}=\left( 0.7,0.2,0.1 \right)\to \left( 0,0,1 \right)$ leads to $B$'s victory. However, the number of troops of $B$ is again lower than one resulting from the optimal strategy. The other two strategies even result in $B$'s failure. Processes and endings of the simulated battles are depicted in Figure \ref{th2}
\begin{figure}[!h]
\centering\includegraphics[width=1.0\linewidth]{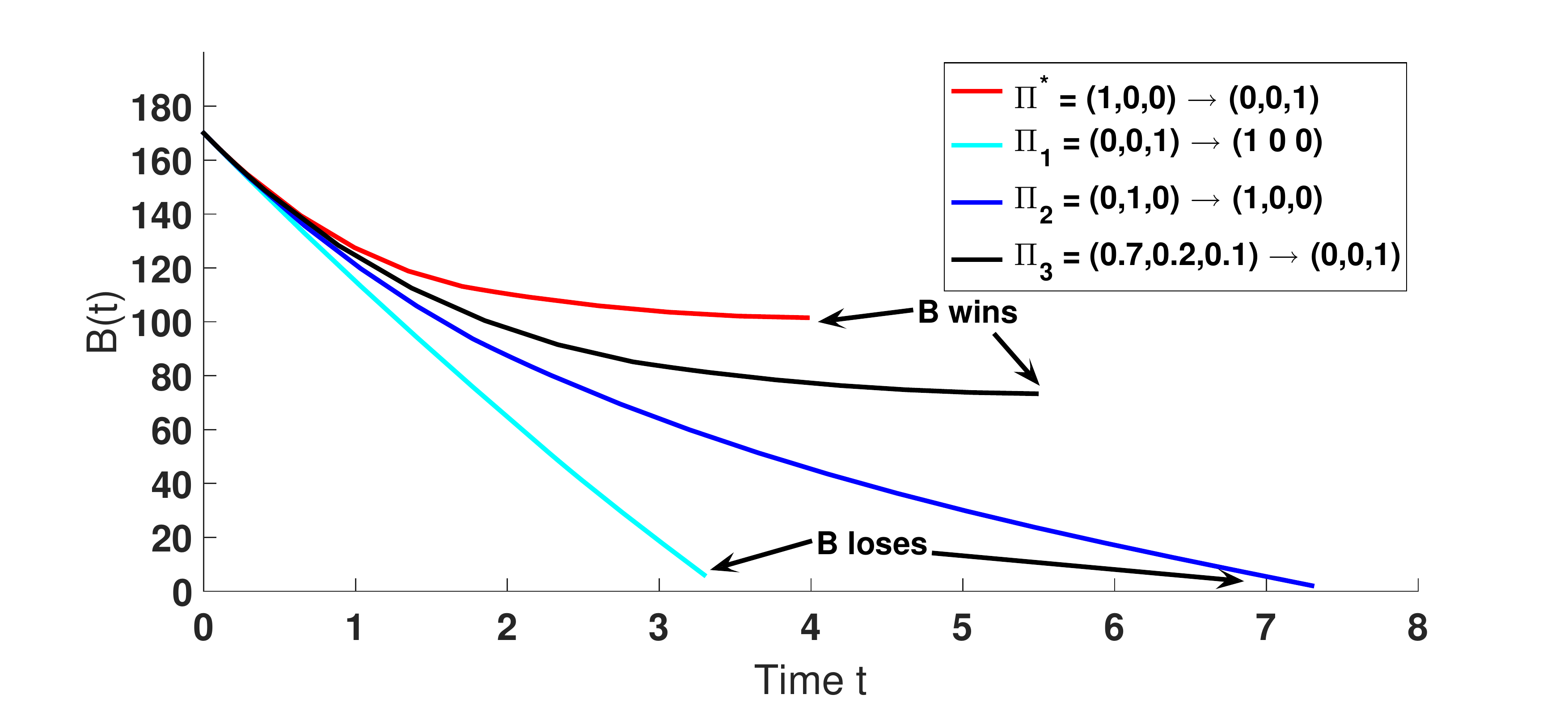}
\caption{Results for Case 2.}
\label{th2}
\end{figure}

\subsection{Case 3}
Now we consider the model with parameters given below:
\begin{table}[!h]
\centering
\begin{tabular}{|c|c|c|c|c|c| }
\hline
$\alpha_c^N$ & $\alpha_d^N$&$\gamma_A$&$\beta_R$&$\beta_N$&$\beta_A$
\\
\hline
$0.4$&$0.2$&$0.6$&$0.5$&$0.2$&$0.5$\\
\hline
\end{tabular}
\caption{Parameters for Case 3}
\end{table}\\
together with initial conditions: ${{B}_{0}}=170;{{R}_{0}}=120;\,\,{{N}_{0}}=60;\,{{A}_{0}}=50.$ The threatening rates are ${{b}_{1}}=0.2;\,{{b}_{2}}=0.08;\,{{b}_{3}}=0.3$. The optimal fire allocation is therefore ${{\Pi }^{*}}=\left( 0,0,1 \right)\to \left( 1,0,0 \right).$ The two contrasting strategies are ${{\Pi }_{1}}=\left( 1,0,0 \right)\to \left( 0,0,1 \right);\,{{\Pi }_{2}}=\left( 0,1,0 \right)$. By ${{\Pi }_{1}}$, $B$ win with more troops lost while ${{\Pi }_{2}}$ leads to $B$'s failure. Processes of the battles are given in Figure \ref{th3}.
\begin{figure}[!h] 
\centering\includegraphics[width=1.0\linewidth]{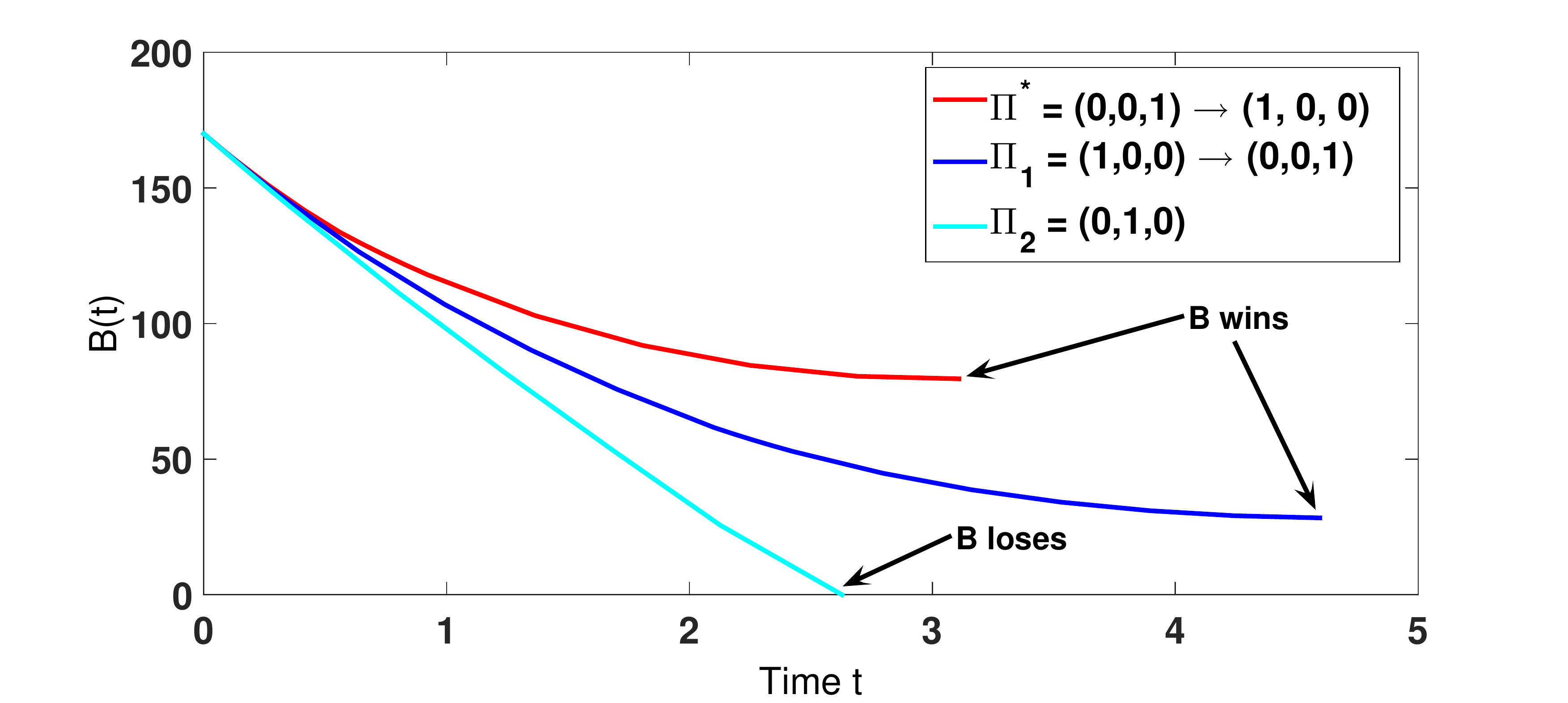}
\caption{Results for Case 3.}
\label{th3}
\end{figure}

 \section{Conclusion}
 A nonlinear Lanchester model of mixed NCW type has been introduced together with the notion of \textit{threatening rates}. Invoking these rates, threats exposed by rivals have been quantified and these quantities help derive the optimal strategy for Blue force, which decreases the losing of its own troops. Numerical results support the theoretical findings.
%\section*{Acknowledgements} 

\enddocument